\documentclass[11pt]{amsart}
\usepackage{amsfonts}
\usepackage[all,cmtip]{xy}
\usepackage{amsmath}
\usepackage{amsthm}
\usepackage{amssymb}
\newtheorem{thm}{Theorem}[section]

\newtheorem*{definition*}         {Definition}
\newtheorem{lemma}[thm]{Lemma}

\theoremstyle{remark}
\DeclareMathOperator{\GL}{GL}
\DeclareMathOperator{\Sp}{Sp}

\newcommand*{\Q}{\mathbb{Q}}
\newcommand*{\Ss}{\mathbb{S}}
\newcommand*{\Qa}{\overline{\mathbb{Q}}}
\newcommand*{\Z}{\mathbb{Z}}
\newcommand*{\G}{\mathbb{G}}
\newcommand*{\A}{\mathbb{A}}
\newcommand*{\R}{\mathbb{R}}
\newcommand*{\N}{\mathbb{N}}

\newcommand*{\C}{\mathbb{C}}
\newcommand*{\Hh}{\mathbb{H}}

\newcommand*{\GQ}{\textrm{Gal}(\overline{\mathbb{Q}}/\mathbb{Q})}
\newcommand*{\Disc}{\textrm{Disc}}

\usepackage{amscd,amssymb,amsmath}
\usepackage{color}

\author{Jonathan pila}
\author{Jacob Tsimerman}
\begin{document}
\title{The Andr\'e-Oort conjecture for the moduli space of Abelian Surfaces}
\maketitle
\begin{abstract}
We provide an unconditional proof of the Andr\'e-Oort conjecture for the coarse moduli space $\mathcal{A}_{2,1}$ of 
principally polarized Abelian 
surfaces,
following the  strategy outlined by Pila-Zannier.
\end{abstract}

\section{introduction and notation}

Set $\Hh_g$ to be the Siegel upper half space 
$$\Hh_g=\{Z\in M_g(\C)\mid Z=Z^t, Im(Z)>0\}.$$
Let $\mathcal{A}_{g,1}$ denote the coarse moduli space of principally 
polarized Abelian varieties of dimension $g$.

Our main theorem is the following, proving the Andr\'e-Oort conjecture for $\mathcal{A}_{2,1}$ :

\begin{thm}\label{AOSP2}
Let $V\in \mathcal{A}_{2,1}$ be an algebraic subvariety, which is equal to the Zariski closure of its CM points. Then $V$ is a special subvariety.
\end{thm}

We follow the general strategy of Pila-Zannier. One ingredient we need is the following Ax-Lindemann-Weierstrass theorem: We 
denote by $$\pi:\Hh_2\rightarrow \mathcal{A}_{2,1}$$ the natural projection map. 
We consider $\Hh_2\subset \R^6$ with coordinates provided by real and imaginary part of the complex coordinates in $\C^3$, 
and call a set  {\it semialgebraic\/} if it is a
semialgebraic set in $\R^6$.
Let $V\subset \mathcal{A}_{2,1}$ be an algebraic variety, $Z=\pi^{-1}(V)$, and 
$Y\subset Z$ an irreducible semialgebraic subvariety of $\Hh_2$. We say that $Y$ is \emph{maximal} if for all semialgebraic subvarieties $Y'$ containing $Y$ with
$Y'\subset Z$, $Y$ is a component of $Y'$. 

\begin{thm}\label{AL} 	
Let $Y\subset Z$ be a maximal semialgebraic variety. Then $Y$ is a weakly special subvariety. 
\end{thm}

Fix $F_g\subset \Hh_g$ to be the standard fundamental domain \cite{Si}. We shall also need the following bound on heights of CM points

\begin{thm}\label{Heights}
Let $x\in F_g$ be a CM point, and let $H(x)$ be the height of $x$. There exists an absolute constant $\delta(g)>0$ such that:
$$H(x)\ll_g |\GQ\cdot \pi(x)|^{\delta(g)}.$$
\end{thm}

We mention that for the strategy to go through, a basic requirement is the definability in $\R_{an,exp}$ of the projection map 
$\pi: F_g\rightarrow \A_{g,1}$ in  the case $g=2$. This was provided (for all $g$) in the recent work \cite{PS} of Peterzil and Starchenko.
We also use another result of these authors \cite{PS1} which says that  a definable, globally complex analytic subset of an
algebraic variety is algebraic. This can be thought of as an analogue of Chow's theorem, and
comes up for us in our proof of \ref{AL}.

The outline of the paper is as follows: in section 2 we review background concerning Shimura varieties. 
In section 3 we prove theorem \ref{Heights}. In section 4 we prove theorem \ref{AL}. 
Our method is different to the one in \cite{P}  in that we do not use the results of Pila-Wilkie, though we do make use of o-minimality.
In section 5 we combine everything to prove our  main theorem \ref{AOSP2}.

\section{Background: Shimura varieties}

We recall here some definitions regarding Shimura Varieties. 
Let $\Ss$ denote the real torus ${\rm Res}_{\C/\R}\G_{m\C}$.
Let $G$ be a reductive algebraic group over $\Q$, and let $X$ denote a conjugacy class of homomorphisms 
$$h:\Ss\rightarrow G_{\R}$$ satisfying the following 3 axioms:

\begin{itemize}

\item The action of $h$ on the lie algebra of $G_{\R}$ only has hodge weights $(0,0),(1,-1)$ and $(-1,1)$ occuring.
\item The adjoint action of $h(i)$ induces a Cartan involution on the adjoint group of $G_{\R}$
\item The adjoint group of $G_{\R}$ has no factors $H$ defined over $\Q$ on which the projection of $h$ becomes trivial.
\end{itemize}

This guarantees that $X$ acquires a natural structure of a complex analytic space. 
Moreover, $G(\R)$ has a natural action on $X$ given by conjugation, and this
turns $G(\R)$ into a group of biholomorphic automorphisms of $X$. We call the pair $(G,X)$ a {\it Shimura Datum\/}. A Shimura datum\/ $(H,X_H)$ is said to be  a {\it Shimura sub-datum\/} of $(G,X)$ if $H\subset G$ and $X_H\subset X$.

Fix $K$ to be a compact subgroup of 
$G(\A_f)$, where $\A_f$ denote the finite Adeles. We then define $$Sh_K(G,X)(\C)=G(\Q)\backslash X\times G(\A_f)/K$$ 
where $G(\Q)$ acts diagonally, and 
$K$ only acts on $G(\A_f)$. It is a theorem of Deligne \cite{D2} that $Sh_K(G,X)$ can be given the structure of an algebraic variety over $\Qa$, and we call $Sh_K(G,X)$ a {\it Shimura variety}. 
 
Given two Shimura varieties $Sh_{K_i}(G_i,X_i),$ and a map of algebraic groups $\phi:G_1\rightarrow G_2$ which takes $X_1$ to $X_2$ and $K_1$ to $K_2$, we 
get an induced map $\tilde{\phi}:Sh_{K_1}(G_1,X_1)\rightarrow Sh_{K_2}(G_2,X_2)$. Given an element $g\in G_2(\A_f)$, right multiplication by $g$ gives a 
correspondence $T_g$ on $Sh_{K_2}(G_2,X_2)$. We
define a {\it special subvariety} of $Sh_{K_2}(G_2,X_2)$ to be an irreducible component of the image of $Sh_{H_1}(G_1,X_1)$ under $T_g\circ\tilde{\phi}$.

We denote by $G^{ad}$ the adjoint form of a reductive group $G$. Following Ullmo-Yafaev \cite{UY}, we make the following definition:
\begin{definition*}\label{weakly}
An algebraic subvariety Z of $Sh_K(G;X)$ is
{\bf weakly special} if there exists Shimura  sub-datum ($H,X_H)$ of $(G,X)$
and a decomposition $$(H^{ad};X_H^{ad}) = (H_1,X_1) \times (H_2,X_2)$$ and a point
$y_2\in X_2$ such that $Z$ is the image of a connected component of $X_1\times{y_2}$.
\end{definition*}	

In this definition, a weakly special subvariety is special iff it contains a special point iff $y_2$ is special.

\section{Heights of CM points}

The aim of this section is to prove that a principally polarized CM Abelian variety in a fundamental domain for $\Hh_2$ has polynomial height in terms of the 
discriminant of its endomorphism algebra, which is essential for us to apply the results of Pila-Wilkie. We only need this result for Abelian surfaces, but we give
the proof in generality since the increased difficulty is only technical, and the theorem is fundamental to the Pila-Zannier strategy. 
For a matrix $Z\in\Hh_g\cap M_g(\Qa)$ we define the height $H(Z)$ to be the maximum of the height of one of the co-ordinates of $Z$ as a point in the affine 
space $\Qa^{g^2}$. To ease notation, we fix the following convention: Given two positive quantities $C_1$ and $C_2$ that may depend
on other variables, we say $C_1$ is \emph{polynomially bounded} in terms of $C_2$ if there are positive constants $a$ and $b$ such that $C_1\leq aC_2^b$.

\begin{thm}\label{PolyHeight}
Let $A$ be a complex, principally polarized Abelian variety with complex multiplication. Set $R=Z(End(A))$ to be the center of its endomorphism algebra and let
$x\in F_g$ be the point representing $A$. There exists a constant $B_g$ depending only on $g$ such that
$H(x)$ is polynomially bounded in terms of $\Disc(R)$.
\end{thm}

Combined with the results of \cite{T}, this proves theorem \ref{Heights}.

\begin{proof}

We first will handle the case where $A$ is simple. $\newline$

\textbf{Case 1: A is simple}$\newline$

In this case, there is a CM field $K$ such that $A$ has $CM$ by $K$. Let $S=\{\phi_1\dots,\phi_g\}$ denote the complex embeddings of $K$ that make up the
CM type of $A$, and set $F$ to be the maximal totally real subfield of $K$, so that $K$ is a quadratic extension of $F$. In this case $R$ is just
the endomorphism ring of $A$. Now, there is a $\Z$-lattice $I\subset O_K$ such that $A$ is isomorphic to $\C^g/\phi_S(I)$ as a complex torus under the embedding
$$\phi_S:K\hookrightarrow \C^g, \phi(a)=(a^{\phi_1},a^{\phi_2},\dots,a^{\phi_g}).$$ Moreover, the order of $I$ must be $R$. 

\begin{lemma}\label{polyideal}
There is a $\nu\in K$ such that $\nu I\subset O_K$ and $[O_K:\nu I]$ is polynomially bounded in terms of $\Disc(R)$. 
\end{lemma}

\begin{proof}

If $R=O_K$, then the lemma is a known consequence of Minkowski's bound. For the general case, set $e_R = [O_K:R]$. Note that $$\Disc(R)=\Disc(O_K)e_R^2$$ and
that $$e_RO_K\subset R\subset O_K.$$ Set $J=O_K\cdot I$, so that $J$ is an $O_K$-ideal with $$e_RJ\subset I\subset J.$$ By the above we can find 
a $\nu\in K$ with $[O_K:\nu J]$ is polynomially bounded in terms of $\Disc(O_K)$. Then $$\nu I\subset \nu J\subset O_K$$ and
$$[O_K:\nu I] \leq [O_K:\nu e_RJ]\leq e_R^{2g}[O_K:J],$$ and the claim follows.

\end{proof}

\begin{lemma}\label{polybasis}
Given a $\Z$-lattice $I\subset O_K$ there is a basis $\alpha_1,\dots,\alpha_{2g}$ of $I$, such that the absolute values of all conjugates of the $\alpha_i$ are
polynomially bounded in terms of $\Disc(O_K)\cdot[O_K:I]$. 
\end{lemma}
\begin{proof}

Consider the standard embedding $\psi$ of $O_K$ as a lattice in $\C^{g}$ given by $$\psi(\alpha)=(\alpha^{\phi_i})_{1\leq i\leq g}.$$
 Then the covolume of $\psi(I)$ as a lattice is $\Disc(O_K)\cdot[O_K:I]$, and every vector in $I$ has norm at least $1$
(since the norm of every algebraic integer is at least 1).  

Now consider the lattice $$\frac{1}{Vol(\C^g/I)^{1/{2g}}}\cdot\psi(I)$$ as an element  $l$ in $SL_{2g}(\Z)\backslash SL_{2g}(\R)$. Let $N,A,O_{2g}$ 
denote the upper triangular subgroup, the diagonal subgroup, and the maximal compact orthogonal group of $Sl_{2g}(\R)$. 
By the theory of Siegel sets, there is a representative $nak$ of $l$ in $N(\R)D(\R)O_{2g}(\R)$ where  $n$ has all its elements bounded by $\frac{1}{2}$ in absolute
value, and the diagonal matrix $d$ has $$d_{1,1}\geq d_{2,2}\geq\dots\geq d_{2g,2g}.$$ Now, since $d_{2g,2g}$ is the norm of an element, we know that 
$$d_{2g,2g}\gg \frac{1}{Vol(\C^g/I)^{1/{2g}}}.$$ Since the product of the $d_{i,i}$ is $1$, we deduce that $d_{1,1}$ is polynomially bounded in terms of 
$\Disc(O_k)\cdot[O_k:I]$. The basis corresponding to $nak$ thus has all its coordinates polynomially bounded in terms of $\Disc(O_K)\cdot[O_K:I]$.
\end{proof}

Now, consider again our Abelian variety $A$. As before, $A$ is isomorphic as a complex torus to $\C^g/\phi_S(I)$ for some $I\subset K$. The principal polarization
on $A$ corresponds to a totally imaginary element $\xi\in K$, which induces the Riemann form $$E_{\xi}(a,b)=tr_{K/\Q}(\xi a b^{\rho})$$ where $\rho$ denotes 
complex conjugation. Moreover, since the polarization is principal $E_{\xi}$ has determinant 1 as a bilinear alternating form on $I$. 

By changing co-ordinates, one can change the pair $(I,\xi)$ to $(I\nu, \xi(\nu\nu^{\rho})^{-1})$ where $\nu\in K^{\times}$. By lemma \ref{polyideal} we can change
$I$ to be a sublattice of $O_K$ with $e_I = [O_K:I]$ polynomially bounded by $\Disc(R)$, so we assume that $I$ is of this form. Next, the determinant of $E_{\xi}$
as a binary form on $I$ is $$N_{K/\Q}(\xi)\cdot \Disc(O_K)[O_K:I]^2,$$ so that $N_{K/Q}(\xi)$ is bounded above by 1. Moreover, $e_IO_K\subset I$, so that
$$tr(e_I^2\xi O_K)\in \Z,$$ and so $\xi\in e_I^{-2}\Disc(K)^{-1} O_K$. 

We know that $-\xi^2$ is a totally positive element, so we can consider the lattice $$L_{xi}=\psi(O_F\cdot (-xi^2)^{\frac14})\subset\R^g.$$ The covolume of 
$L_{xi}$ is polynomially bounded in terms of $\Disc(R)$, and must contain an element inside a sphere with radius polynomially bounded in terms of the covolume. Therefore,
there exists an element $\nu\in O_F$ such that $\nu\nu^{\rho}\xi=\nu^2\xi$ has all its conjugates polynomially bounded in terms of $\Disc(R)$. Since $\nu$ must
have norm polynomially bounded by $\Disc(R)$, we can
and do assume that $I\subset O_K$ with $[O_K:I]$ polynomially bounded by $\Disc(R)$, and that $\xi$ has all its conjugates polynomially bounded by $\Disc(R)$.

Now consider the representative $(I, \xi)$ of $A$.  To pick a symplectic basis of $(\nu^{-1} I)$, we simply take a basis $\alpha_i$ of
$I\cap O_F$ as in lemma \ref{polybasis}. Next we consider the lattice $Im(\psi_s(I))$ in $(i\R)^g$. Pick a 
basis as in lemma \ref{polybasis}, and refine it to the dual symplectic basis $\beta'_i$ to $\alpha_i$. Since all the conjugates of $\xi$ are polynomially
bounded in terms of $\Disc(R)$, the basis $\beta'_i$ has all its components polynomially bounded in terms of $\Disc(R)$ as well. Lift $\beta'_i$ to 
elements $\beta_i = \beta'_i + c$, where $c$ is an element in $F$, such that $\beta_i\in I$. Note that $c$ is an element of $e_I^{-1}O_F/O_F$, and can thus be
chosen to have all its components polynomially bounded by $\Disc(R)$. Finally, since the values $E_{\xi}(\beta_i,\beta_j)$ might not be $0$, we replace
$\beta_i$ by $\beta_i - \sum_{i\leq j}E_{\xi}(\beta_i,\beta_j)\alpha_j$. 

Now consider the matrix $Z\in \Hh_g$ which represents the elements $\beta_i$ in terms of the $\alpha_i$. $Z$ is the matrix representing $A$. Moreover,
the above construction gives $Z=X + iY$, where $X\in M_g(\Q)$ with all its denominators polynomially bounded by $\Disc(R)$, and $Y\in M_g(K)$ such that
all the entries of $Y$ have all their complex conjugates polynomially bounded by $\Disc(R)$, and likewise for the denominators of the entries of $Y$
(the denominator of an algebraic number $\alpha$ is the smallest integer $n$ with $n\alpha$ an algebraic integer).

It is evident that $Z$ has height polynomially bounded by $\Disc(R)$, and that the degree of all the entries in $Z$ is at most $4g$. All that is left to see is
that putting $Z$ in its fundamental domain does not increase the height by too much, and so the following lemma completes the proof:

\begin{lemma}

Let $g\in\N$ be a natural number, and $Z=X+iY$ be an element in siegel upper half space $\Hh_g(\C)$. Set $h(Z) = Max(|z_{ij}|, \frac1{|Y|})$. Then for $\gamma\in 
Sp_{2g}(\Z)$ such that $\gamma\cdot Z\in F_g$, all the co-ordinates of $\gamma$ are polynomially bounded in terms of $h(Z)$. 

\end{lemma}

\begin{proof}

Set $\gamma=\begin{pmatrix} A & B\\ C & D\end{pmatrix}.$ 
The proof involves going through and quantifying Siegel's proof in \cite{Si} of the fact that $F_g$ is a fundamental domain. As is well known, 
$C,D$ are such that $|det(CZ+D)|$ is minimal, and by perturbing $Z$ slightly we can ensure that this determines $C$ and $D$ up to the left action of 
$\GL_g(\Z)$. Now, as in Siegel pick an element $U\in\GL_g(\Z)$ such that $UY^{-1}U^t$ is  Minkowski reduced, and set 
$$\gamma_0=\begin{pmatrix} A_0 & B_0\\ C_0 & D_0\end{pmatrix}=\begin{pmatrix} U & 0\\ 0&U^{-t}\end{pmatrix}\gamma$$ and $$\gamma_0\cdot Z=X_0+iY_0.$$

Now, set $y_1,y_2,\dots,y_n$ to be the diagonal elements of $Y_0^{-1}$, and $c_l,d_l$ to be the rows of $C_0,D_0$. By \cite{Si}, 
pg. 40, eq. (87), (88) we have 

\begin{equation}\label{Y1}
y_l = Y^{-1}[Xc_l + d_l]+Y[c_l]
\end{equation} and 
\begin{equation}\label{Y2}
\prod_{i=1}^n y_i\ll |Y|^{-1}.
\end{equation} Now, the eigenvalues of $Y$ are polynomial bounded in terms of the co-ordinates
of $Y$, and their product is equal to the determinant of $Y$, hence the inverses of the eigenvalues of $Y$ 
are polynomially bounded in terms of $h(Z)$. Thus,
since integer vectors have euclidean norm at least 1,  equation (\ref{Y1}) implies that $y_l$ is polynomially bounded below by $h(Z)$, 
which is to say that 
$y_l^{-1}$ is polynomially bounded in terms of $h(Z)$.  But  now equation (\ref{Y2}) implies that $y_l$ is polynomially bounded in terms of $h(Z)$, and thus so
are the norms of $c_l$ and $d_l$. Hence all the co-ordinates of $C_0$ and $D_0$ are polynomially bounded in terms of $h(Z)$.

Now, we can find $A_1,B_1$ with polynomially bounded entries such that the matrix $$\gamma_1:=\begin{pmatrix}A_1& B_1\\ C_0 & D_0\end{pmatrix}.$$ is in
$\Sp_4(\Z)$. Set $Z_1=\gamma_1\cdot Z$. There is then an upper triangular matrix $\gamma_2=\gamma\gamma_1^{-1}$ which takes $Z_1$ into $F_g$. 
The lemma now follows from lemma \ref{minkowski}, which is well known but we include for completeness.

\end{proof}

\begin{lemma}\label{minkowski}
Let $U\in\GL_g(\Z)$ be such that $UYU^t$ is Minkowski reduced. Then $U$ is polynomially bounded in terms of $h(Y)=MAX(|y_{ij}|, \frac{1}{|Y|})$.
\end{lemma}

\begin{proof}

Let $Y'=UYU^t$, and set $y_1,\dots,y_n$ to be the diagonal elements of $Y'$. Letting $u_l$ be the rows of $U$, we have $Y[u_l]=y_l$, and as before

\begin{equation}\label{Y3}
\prod_{i=1}^n y_i\ll |Y'|=|Y|.
\end{equation}

Now, since the $u_l$ have euclidean norm at least 1, we have that the $y_l^{-1}$ are polynomially bounded in terms of $h(Y)$, and hence by equation \ref{Y3}
the $y_l$ are polynomially bounded in terms of $h(Y)$. Thus we conclude that the euclidean norms of the $u_l$ are polynomially bounded in terms of $h(Y)$,
 which implies the lemma.

\end{proof}

$\newline$\textbf{Case ii: General A} $\newline$

In general, there are simple Abelian varieties $A_i$ of dimension $g_i$ with complex multiplication by $K_i$ of CM type $S_i$  such that $A$ is isogenous to 
$$\prod_i A_i^{n_i}$$ so that $\sum_i n_ig_i = g$. We assume that the types $(K_i,S_i)$ are inequivalent (which does NOT mean that the fields $K_i$ are 
all distinct!) so that $$End(A)\subset \prod_i M_{g_i}(K_i)$$ and $$R\subset \oplus_{i} O_{k_i}.$$  For simplicity of notation we set $O_A=\oplus_{i} O_{k_i}$.
As before, define $e_R$ to be the index of $R$ in $O_A$ and so that $$\Disc(R)=e_R^2\cdot\prod_i \Disc(k_i).$$

There is an embedding $$\psi_S : \oplus_i K_i^{n_i}\rightarrow \C^g$$ given by $\psi_S=\oplus_i \psi_{S_i}^{n_i}$, and a lattice $I\subset \oplus_i K_i^{n_i}$ such
that $$A(\C)\cong  \C^{g}/\psi_S(I).$$ Moreover, $I$ is invariant under multiplication by $R$. Consider $J=O_A\cdot I$. Then $J$ is an $O_A$ module with
$$e_R\cdot J\subset I\subset J.$$ We thus have a direct sum decomposition $$J=\oplus_i J_i$$ with $J_i$ an $O_{K_i}$ ideal, and so in fact we can decompose further
$$J_i = \oplus_{j=1}^{n_i} P_{ij}$$ with each $P_{ij}$ an $O_{K_i}$ module. By scaling with elements of $K_i^{\times}$, we can guarantee that
$P_{ij}\subset O_{k_i}$ of index at most $\Disc(K_i)^{1/2}$. Therefore, there is an integer $N$ polynomially bounded in terms of $\Disc(O_A)$ such that
the ring $$\oplus_i\left(O_{K_i} + N\cdot M_{n_i}(O_{K_i})\right)$$ preserves $J$, and thus the ring
$$\Z+Ne_R\oplus_i\cdot M_{n_i}(O_{K_i})$$ preserves $I$. 

The following lemma allows us to reduce to the case where $A\cong A_1^{n_1}$. 

\begin{lemma}\label{generalab}

There are principally polarized Abelian varieties $B_i$ isogenous to $A_i^{n_i}$, and an isogeny $\lambda:A\rightarrow\oplus B_i$ compatible
with polarizations such that $\lambda$ has degree polynomially bounded in terms of $\Disc(O_A)$.

\end{lemma} 

\begin{proof}
Consider $I_0 = O_A\cdot I$, and let $A^0$ be the Abelian variety whose complex points are $$\C^{g}/\psi_S(I_0).$$ Thus $A$ has an isogeny  $\lambda^0$ 
to $A^0$ of degree polynomially bounded in terms of $\Disc(O_A)$. Since $A^0$ has an action by $O_A$, it splits as $A^0\cong\oplus_i A^0_i$, where
$A^0_i$ has an action of $O_{K_i}$, and has  CM type $(K_i,S_i)$. 
The polarization on $A$ induces a polarization $\eta$ on $A^0$ via $\lambda^0$ of degree polynomially bounded in terms of $\Disc(O_A)$, and as the $A^0_i$ have
no non-trivial homomorphisms between them, $\eta$ splits as $\eta\cong\oplus_i \eta_i$. There is then an isogeny from $A^0_i$ to $B_i$ of degree
 $deg(\eta)^{\frac12}$ such that $B_i$ is principally polarized. Composing with $\lambda^0$ completes the  proof.
 \end{proof}
 
 If $Z_1,Z_2$ are 2 points in the fundamental domain $F_g$ corresponding to principally polarized Abelian varieties with an isogney of degree $C$ between them,
 then  $\frac{H(Z_1)}{H(Z_2)}$ is polynomially bounded in terms of $C$. Thus, by lemma \ref{generalab} we can and do restrict to the case where $A\cong A_1^{n_1}$. 
 
 Now, as before we can and do assume that $I\subset O_{K_1}^{n_1}$, $I$ is invariant by 
$\Z+ Ne_R M_{n_1\times n_1}(O_{K_1})$, where $N$ is an integer polynomially bounded in terms of $\Disc(K_1)$, 
and $[O_{K_1}^{n_1}:I]$ is polynomially bounded in terms of $\Disc(R)$.
Then the polarization on $A$ is given by a matrix $E\in M_{g\times g}(K_1)$ satisfying $E=-E^*$, where $E^*$ denotes the transpose-conjugate matrix.
As the symplectic form defined by $E$ takes integer values on $I$, we deduce that there is an integer polynomially bounded by $\Disc(R)$, which we may
take to be $N$, such that $NE$ has entries which are algebraic integers. Moreover, as $E$ defines a principal polarization of $A$, the determinant of $E$
is polynomially bounded in terms of $\Disc(R)$.

Moreover, as before let $\zeta\in O_{K_1}$ be a totally imaginary element, with entries polynomially bounded by $\Disc(K_1)$ and $-i\phi(\zeta)>0$ for all 
$\phi\in S_1$. Then the quadratic form $Q$ on $K_1^n$ defined by $$Q(v_1,v_2) = tr_{K_1/\Q}(\zeta\cdot v_1 E v_2^*)$$ is positive definite.

\begin{lemma} \label{hermitian}

There exists an invertible matrix $g\in M_{g\times g}(O_{K_1})$ such that $gEg^{*}$ has entries all of whose conjugates are
polynomially bounded in terms of $\Disc(R)$. 

\end{lemma}
\begin{proof}

Consider the quadratic form  $$Q(v_1,v_2) = tr_{K_1/\Q}(\zeta\cdot v_1 E v_2^*)$$ as a positive-definite quadratic form on ${O_K}^{n_1}$ thought of as
as $\Z^{n_1\cdot[K:\Q]}$. Since $NE$ has entries which are algebraic integers, the smallest non-zero value $Q$ can take is $\frac{1}{N}$.
By repeating the proof of \ref{polybasis}, we can find a basis of $n_1\cdot[K:\Q]$ elements $v_i$ in $O_K^{n_1}$ such that
$Q(v_i,v_i)$ is polynomially bounded by the determinant of $Q$, which is in turn polynomially bounded in terms of $\Disc(R)$. Pick a subset
 $$\{w_j,1\leq j\leq n_1\}$$
of the $v_i$ which are linearly independant over $K$, and make them the rows of $g$. For $\phi\in S_1$, Consider the positive definite matrix
$E_\phi:=\phi(\zeta\cdot gEg^{*})$. By construction, $E_{\phi}$ is hermitian, positive definite, and has diagonal elements polynomially bounded in terms of
$\Disc(R)$. Thus all the enties are automatically polynomially bounded in terms of $\Disc(R)$. As the conjugates of $\zeta^{-1}$ are also polynomially
bounded in terms of $\Disc(R)$, this completes the proof.
\end{proof}

Take $g$ as in lemma \ref{hermitian}, and note that $g$ must have determinant polynomially bounded in terms of $\Disc(R)$. We can thus replace $(I,E)$ by
$(g^{-1}(I), gEg^*)$. Now we can pick a basis for $g^{-1}(I)$ of vectors whose entries have all their conjugates and 
denominators polynomially bounded in terms of $\Disc(R)$ as in lemma \ref{polybasis}. The rest of the proof follows as in case i.

\end{proof}

\section{Ax-Lindemann-Weierstrass}
The aim of this section is to prove theorem \ref{AL}. We begin by showing that we can restrict our attention to complex algebraic varieties:

\begin{lemma}\label{semical}
Let $Z$ be a complex analytic submanifold of $\C^n$, and $W\in Z$ be a maximal irreducible semi-algebraic set. Then $W$ is a subset of a
 complex analytic subvariety of $\C^n$ of the same dimension as $W$.
 \end{lemma}
 
 \begin{proof}
 
Take $U$ to be the zariski closure of $W$, and let $O\in W$ be a smooth point of $X$. Let $m=\dim U=\dim W$. Let $z_1,z_2,\cdots,z_n$ be the usual co-ordinates 
on $\C^n$, with $x_j,y_j$ being real co-ordinates such that $z_j=x_j+iy_j$ as usual. We first want to 'complexify' $U$ into a complex 
variety inside $\C^n$.  Since $U$ is a real algebraic variety over $\R$, we consider the set of its complex points $U(\C)$ as an abstract complex algebraic variety. Moreover,
the inclusion map $i:U\rightarrow\R^{2n}$ is given by $n$ pairs of polynomial maps $(f_i,g_i)$ from $U$ to $\R$, so that 
$i(u)=(f_1(u),g_1(u),\dots,f_n(u),g_n(u))$. Thus we can consider the complexified map $i_{\C}:U(\C)\rightarrow\C^{2n}$ via 
$$i_{\C}(u) = (f_1(u)+ig_1(u),\dots,f_n(u)+ig_n(u)).$$ The map $i_{\C}$ is the identity map on the real points $U(\R)$, 
and its image on the whole of $U(\C)$ is a complex algebraic variety\footnote{For those familiar with Weil restriction, this simply reflects that Weil restriction
is the right adjoint to the base change functor.}. 

Now, Pick local \emph{real} co-ordinates $u_1,\dots,u_m$ for $U$ around $O$, so that the $u_i$ become \emph{complex} co-ordinates for $U(\C)$ around $O$.
Define $Y$ to be the pullback of $Z$ along $i_{\C},$ so that $$Y:=i_{\C}^{-1}(Z\cap i_{\C}U(\C)).$$ Then $O\in Y$, and locally around $O$ 
in the co-ordinates $u_i$, $Y$ is a complex manifold which contains $\R^m$.  Since $Y$ is a complex manifold its tangent space at $O$ is a complex subspace,
and since it contains $\R^m$ it must be all of $\C^m$. Thus $Y$ contains an open neighbourhood of $U(\C)$, and thus $Z$ contains an open neighbourhood of 
$i_{\C}(U(\C))$. Since $W$ was assumed to be maximal, $W$ must be of the same dimension as $i_{\C}(U(\C))$, and this completes the proof.

\end{proof}

We shall make use the following 2 lemmas:

\begin{lemma}\label{UY}
Suppose $W\in \mathcal{A}_{2,1}$ is an algebraic variety such that $\pi^{-1}(W)$ has an algebraic component. Then $W$ is weakly special.
\end{lemma}

\begin{proof}
This is the main theorem of \cite{UYS}.
\end{proof}

Here, $Y$ is an algebraic subvariety of $\Hh_2$ if $Y=\tilde{Y}\cap\Hh_2$ for some algebraic subvariety $\tilde{Y}\subset \C^3$ 
\cite{UYS}. 
In view of lemma 4.1, the condition that  a component of $\pi^{-1}(W)$ is algebraic 
is equivalent to it being semialgebraic.

\begin{lemma}\label{defineableanalytic}
Let $W$ be a complex algebraic variety, and $D\in W$ be a definable,  globally complex analytic subset. Then $D$ is algebraic.
\end{lemma}	

\begin{proof}

By taking an affine open set in $W$, it suffices to consider the case where $W$ is an affine subset of projective space.
Now, one can express $W$ as $M\backslash E$ where $M$ is a projective variety and $E$ is an algebraic subvariety of $M$.
Theorem 5.3 in \cite{PS1} then applies that the closure of $D$ in $M$ is a definable, globally analytic subset of $M$, and thus $D$ must be algebraic
by Chow's theorem.
\end{proof}

\textbf{Proof of Theorem  \ref{AL}}: First, by lemma \ref{semical} we can assume that the Zariski closure of $Y$ is complex algebraic.
Note also that if $dim Z = dim Y$, then $Z$ must equal $Y$ and be semialgebraic itself, so that we are done
by lemma \ref{UY}. We can thus assume that $dim Z = 2$ and $dim Y = 1$. 

Define $Z^0$ to be the connected component of $Z$ containing $Y$. Now consider a fundamental domain $F_2$ which intersects $Z^0$ and define $Z_0=Z\cap F_2$. 
We know that $Z_0$ is definable in $\R_{an,exp}$ by the main result of \cite{PS}. From now on
we say definable to mean definable in $\R_{an,exp}$. 


We define

$$X = \{g\in SP_{2g}(\R)\mid dim_{\C}(g\cdot Y\cap Z_0) = 1\}.$$

$X$ is a definable subset of $SP_{2g}(\R)$. Moreover, set $\Gamma_0\subset\Gamma$ to be the monodromy group of $V$, so that $\Gamma_0$ preserves a connected 
component of $Z$. Then for all elements of $g\in\Gamma_0$ such that $Y\cap gF$ is not empty, we must have $g\in X$. 

If $V$ is not hodge-generic in $\mathcal{A}_{2,1}$, it must be contained in a 2-dimensional special subvariety, which would mean that $Z$ is special, contradicting the 
maximality of $Y$. Therefore $V$ is hodge-generic, and so $\Gamma_0$ is Zariski dense in $Sp_4(\R)$.

Now, since $X$ is definable it admits an analytic cell decomposition \cite{DM}. Thus $X$ is a union of finitely many irreducible, definable components 
$$X=\cup_{i=1}^m X_i,$$ such that each $X_i$ is real-analytically homeomorphic to an open ball of some dimension. Note that some of the $X_i$ may be points.
By analytic continuation, we have $X_i\cdot Y\subset Z$ for all $1\leq i\leq m$. 
\subsection{ Case 1: $\forall 1\leq i\leq m, \dim_{\R} X_i\cdot Y = 2$}

Since everything is locally real analytic, we must have $\forall 1\leq i\leq m, X_i\cdot Y = x_i\cdot Y$, where $x_i\in X_i$ is an arbitrary point.

\begin{lemma}\label{closed-image}
Under the assumptions above, $\pi(Y)$ is an algebraic subvariety of $V$.
\end{lemma}

\begin{proof} 

We have that $$\pi(Y) = \cup_{g\in\gamma} \pi(Y\cap g\cdot F_2) = \cup_{g\in\gamma} \pi(gY\cap F_2).$$

Now, if $g\in\gamma$ and $gY\cap F_2\neq 0$, then in fact $$Y\cap F_2 \subset F_2\cap g\cdot Z = Z_0$$ and so $g\in X$. Thus, there exists an $i$ with $g\in X_i$. 
We thus have that $$\pi(Y) = \cup_{1\leq i\leq m} \pi(x_i\cdot Y \cap F_2),$$ and thus $\pi(Y)$ is a finite union of closed sets, and therefore closed. It is also
definable, and so must be algebraic by lemma \ref{defineableanalytic}.

\end{proof}

We now have that $Y$ is a semialgebraic subvariety such that $\pi(Y)$ is also algebraic. By lemma \ref{UY}, $Y$ must be special. 

\subsection{ Case 2: For some $1\leq i\leq m, \dim_{\R} X_i\cdot Y>2$}

Wlog, we assume $\dim_{\R} X_1\cdot Y>2$. Take a small real analytic curve $I\subset X_1$, and consider a local complexification $I_{\C}\subset Sp_4(\C)$. 
Define $Y^0$ to be a connected component of $Y_2=I_{\C}\cdot Y\cap \Hh_2$ . By analyticity, $Y^0$  is contained in $Z$. Moreover, the complex dimension of 
$Y^0$ must be at least 2, and so $Y_2$ is an open component of $Z$. Since $I_{\C}$ is definable, $Y_2$ is also definable. Now, define

$$X_2:=\{g\in Sp_4(\R)\}\mid \dim_{\C} g\cdot Y_2\cap Z_0 = 2\}.$$

Note that for any point $g\in X_2$, we must have $g\cdot Z^0 = Z^0$. We now prove that $X_2\cap\Gamma$ is infinite. Assume not.
Since $X_2\cap\Gamma$ is finite, then $I\cdot Y$ intersects only finitely many fundamental domains. Pick $p\in I$, so that $p\cdot Y$ intersects finitely many
fundamental domains, and hence by lemma \ref{UY} $p\cdot Y$ is a weakly special variety. But weakly special subvarieties are invariant by infinitely many 
elements of $\Gamma$ and hence intersect infinitely many fundamental domains. This contradiction proves that $X_2\cap\Gamma$ is infinite.

Since $X_2$ is also definable, it must contain a real analytic curve $U\subset X_2$. Consider now the group 
$$G_Z=\{g\in Sp_4(\R)\mid g\cdot Z^0 = Z^0\}.$$

Thus $G_Z$ contains a 1-parameter subgroup,
and so the lie algebra $lie(G_Z)$ is a positive-dimensional vector space. Moreover, since $\Gamma_0\subset G_Z$, we must have 
$lie(G_Z)$ is invariant under conjugation by $\Gamma_0$, and therefore also by the Zariski closure of $\Gamma_0$. Thus $lie(G_Z)$ is invariant 
under conjugation by $Sp_4(\R)$. Since $Sp_4(\R)$ is simple, this means that $lie(G_Z)=lie(Sp_4(\R))$, and so $G_Z=Sp_4(\R)$, which is a contradiction.

\section{Proof of Theorem \ref{AOSP2}}

\begin{proof}

Let $V\subset \mathcal{A}_{2,1}$ be a variety defined over some number field $K$, which is the closure of the CM points inside it. Consider $Z=\pi^{-1}(V)$ and 
let $Z_0=F_g\cap Z$. $Z_0$ is definable. Now, let $x\in V$ be a CM point and set $\{x_i\}_{1\leq i\leq m}$ to be the galois orbit of $x$ over $K$.
Set $w_i\in Z_0$ be a pre-image of $x_i$, so that $\pi(w_i)=x_i$. By Theorem \ref{Heights}, there is an $\delta(g)>0$ such that the heights of the $x_i$ are at 
most $$H(w_i)\ll m^{1/\delta_g}.$$ Thus, we can conclude by Pila-Wilkie (\cite{PW} theorem 1.8) that at least 1 (in fact, most, but all we need is 1) 
$x_i$ is contained in a positive dimensional algebraic variety. By Theorem \ref{AL} this must be a weakly special subvariety. 
Thus all but finitely many CM point in V must have a Galois conjugate which is contained in a positive-dimensional weakly special subvariety of $V$. 
 
Since Galois conjugates of weakly special subvarieties are weakly special, we conclude that all but
finitely many CM points lie on positive dimensional weakly special subvarieties $S_i$ of $V$, which are then special by virtue of containing CM points. 
If $V$ is 1 dimensional, than any special subvariety that $V$ contains must in fact 
equal $V$. So we assume from now on that the dimension of $V$ is 2, and each of the $S_i$ has dimension 1. Assume for the sake of contradiction that
$V$ is not special. 

Now, say $S_i$ is a weakly special subvariety. Then there exist a semisimple subgroup $H_i\subset Sp_4(\R)$ and an element $z\in F_g$ such that 
$S_i = \pi(H_i\cdot z)$. 

\begin{lemma}\label{finitespecial}

The set of groups $H_i$ is finite.

\end{lemma}

\begin{proof}
There are finitely many semisimple lie algebras which embed into $lie(Sp_4(\R))$, and by lemma A.1.1 in \cite{EMV}
these come in finitely many sets of conjugacy classes, so we can assume wlog that there is a fixed semisimple lie group $H\subset Sp_4(\R)$ and elements
$t_i\in Sp_4(\R)$ with $H_i=t_iHt_i^{-1}$. Now, as $S_i$ is a special subvariety, the group $\Gamma_i=H_i\cap Sp_4(\Z)$ is Zariski dense in $H_i$. Since
$\Gamma_i$ is also finitely generated, the set of such groups is countable and hence the set of possible $H_i$ is countable. 

Now, consider the set $$B=\{((t,z)\in Sp_4(\R)\times F_g \mid tHt^{-1}\cdot z\subset Z^0\},$$ which is definable. If $(t,z)\in B$, then either
$tHt^{-1}z$ is special, or by theorem \ref{AL} it must be contained in a special variety. But by dimension considerations, that special variety must have dimension
at least 2, and so it must be $V$. Since we're assuming that $V$ is not special, we conclude that the special subvarieties $S_i$ are precisely the images
of $tHt^{-1}\cdot z$ for $(t,z)\in B$. 

 Since a countable definable set is finite, this proves the claim.
\end{proof}

By lemma \ref{finitespecial} there are finitely many groups $H_1,\dots,H_m$
such that every weakly special subvariety contained in $Z$ which intersects the upper half plane is an $H_i$ orbit. Define $U$ to be the pre-image of all
weakly special subvarieties in $V$ restricted to the fundamental domain $F_g$ so that by the above
$$U = \{w\in Z_0\mid \exists i\in\{1,2,\dots,m\}, H_i\cdot w\subset Z\}.$$

We therefore have that $U$ is definable. Moreover, since $U$ is contained in $Z^0$ its dimension is everywhere locally at most 2. Moreover, since $U$ cannot be 
a finite union of weakly special subvarieties of dimension 1, it dimension must somewhere be $2$. 
Now, let $W_i,i\in\N$ denote the countably many special subvarieties of  $\mathcal{A}_{2,1}$ which have dimension at most 2.
Every weakly special subvariety of $\mathcal{A}_{2,1}$ is contained in one of the $W_i$, so we know that 
$$U=\cup_{i\in\N} U\cap \pi^{-1}(W_i).$$ But now, if $V$ is not special than $$\cup_{i\in\N}V\cap W_i$$ is a countable union of algebraic varieties of dimension
at most 1. Since $U$ must somewhere have dimension 2, this is a contradiction. 

\end{proof}

\end{document}